\documentclass[12pt, reqno]{amsart}
\usepackage{amsmath, amsthm, amscd, amsfonts, amssymb, graphicx, color, mathrsfs,enumitem}
\usepackage[all]{xy}
\usepackage{slashed}
\usepackage{mathabx}
\usepackage{tipa}
\usepackage{soul}
\usepackage{cancel}
\usepackage{ulem}
\usepackage{mathtools}

\newcommand{\M}{\mathcal{M}}

\usepackage{xcolor}
\usepackage[colorlinks,citecolor=blue,hypertexnames=false]{hyperref}

\newtheorem{theorem}{Theorem}[section]
\newtheorem{lemma}[theorem]{Lemma}

\theoremstyle{definition}
\newtheorem{definition}[theorem]{Definition}

\theoremstyle{remark}
  
\numberwithin{equation}{section}

\newcommand{\half}{\frac{1}{2}}
\newcommand{\Ra}{\mathbb{R}}

\newcommand{\Rn}{{\mathbb{R}}^{n}}
\newcommand{\Rnn}{{\mathbb{R}}^{2n}}
\newcommand{\ene}{{\mathbb{N}}}
\newcommand{\An}{\mathcal{A}_{k,\,\ell}}

\newcommand{\Ann}{\mathcal{A}_{1,\, 1}}

\begin{document}
\setcounter{page}{1}

\title[Anharmonic semigroups and applications]{Anharmonic semigroups and applications to  global well-posedness of nonlinear heat equations}

\author[D. Cardona]{Duv\'an Cardona}
\address{
  Duv\'an Cardona:
  \endgraf
  Department of Mathematics: Analysis, Logic and Discrete Mathematics
  \endgraf
  Ghent University, Belgium
  \endgraf
  {\it E-mail address} {\rm duvanc306@gmail.com, duvan.cardonasanchez@ugent.be}
  }

  \author[M. Chatzakou]{Marianna Chatzakou}
\address{
  Marianna Chatzakou:
  \endgraf
  Department of Mathematics: Analysis, Logic and Discrete Mathematics
  \endgraf
  Ghent University, Belgium
  \endgraf
  {\it E-mail address} {\rm Marianna.Chatzakou@UGent.be}
  }

  \author[J. Delgado]{Julio Delgado}
\address{
  Julio Delgado:
  \endgraf
  Departmento de Matematicas
  \endgraf
  Universidad del Valle
  \endgraf
  Cali-Colombia
  \endgraf
    {\it E-mail address} {\rm delgado.julio@correounivalle.edu.co}}

\author[V. Kumar]{Vishvesh Kumar}
\address{
  Visvesh Kumar:
  \endgraf
  Department of Mathematics: Analysis, Logic and Discrete Mathematics
  \endgraf
  Ghent University, Belgium
  \endgraf
  {\it E-mail address} {\rm  Vishvesh.Kumar@UGent.be}
  }

\author[M. Ruzhansky]{Michael Ruzhansky}
\address{
  Michael Ruzhansky:
  \endgraf
  Department of Mathematics: Analysis, Logic and Discrete Mathematics
  \endgraf
  Ghent University, Belgium
  \endgraf
 and
  \endgraf
  School of Mathematics
  \endgraf
  Queen Mary University of London
  \endgraf
  United Kingdom
  \endgraf
  {\it E-mail address} {\rm michael.ruzhansky@ugent.be}
  }


\date{\today}
\thanks{The authors are supported by the FWO Odysseus 1 grant G.0H94.18N: Analysis and Partial Differential Equations and by the Methusalem programme of the Ghent University Special Research Fund (BOF) (Grant number 01M01021). The last two authors are supported by  FWO Senior Research Grant G011522N.   Julio Delgado is also supported by Vic. Investigaciones Universidad del Valle CI-71352. Michael Ruzhansky is also supported by EPSRC grants EP/R003025/2 and EP/V005529. Duv\'an Cardona is a postdoctoral fellow of the Research Foundation-Flanders
(FWO) under the posdoctoral
grant No 1204824N. Marianna Chatzakou is a postdoctoral fellow of the Research Foundation-Flanders
(FWO) under the postdoctoral grant No 12B1223N}

\begin{abstract} In this work we consider the  semigroup $e^{-t\An^{\gamma}}$ for $\gamma>0$ associated to an { \it anharmonic oscillator} of the form 
$ \An=(-\Delta)^{\ell}+|x|^{2k}$ where $k,\ell$ are integers $\geq 1$. By introducing a suitable H\"ormander metric on the phase-space we  analyse the semigroup $e^{-t\An^{\gamma}}$  within the framework of H\"ormander  $S(M,g)$ classes and obtain mapping properties in the scale of modulation spaces $\M^{p,q},\, 0<p,q\leq \infty,$ with respect to an anharmonic modulation weight. 
As an application, we apply the obtained bounds to establish the well-posedness for the nonlinear heat equation associated with $\An^{\gamma}$. It is worth noting that the results presented in this paper are novel, even in the case where $\gamma=1.$
\end{abstract} \maketitle

\allowdisplaybreaks
\section{Introduction}
Let $k,\ell$ be integers $\geq 1.$ In this paper we consider the operator $\An:=(-\Delta)^{\ell}+|x|^{2k}$ on $\Rn$ and the corresponding semigroup $e^{-t\An^{\gamma}}$ for $\gamma>0$. The operators  $\An$   are known as {\it anharmonic oscillators } and they play an important role in quantum physics, particularly when modelling vibrations for suitable families of molecules. Relevant  cases such as the quartic, quartic-quadratic and sextic potentials $V(x)$ associated to the one dimensional Hamiltonian $-\frac{d^2}{dx^2}+V(x)$ are of interest in the analysis of specific quantum systems of molecules. For instance in the analysis of water molecules, under some circumstances, the harmonic oscillator fails to be even an approximate description. Therein a successful model is the quartic-quadratic oscillator given by a potential $V(x)$ of the form $ax^4+bx^2$, where $a, b$ are suitable constants. This model is, in particular, useful for treating low-frequency vibrations in small ring molecules \cite{Zb1}. Indeed, there is a bulk of recent literature on the anharmonic oscillators of quartic and sextic order in the study of molecular vibrations (cf. \cite{Babenko1, Elkamash, Romat1} and the references therein). We point out that, we can also consider more general operators than $\An$ allowing lower order terms with respect to each one of the symbols $|\xi|^{2\ell}$ and $|x|^{2k}$. Such operators were considered in \cite{anh:cdr2}. Furthermore, more generally we consider  positive fractional powers $\An^{\gamma}$  of our anharmonic oscillators. We associate to the operator $\An$ a suitable  H\"ormander metric $g=g^{k, \ell}$ on the phase-space,   an appropriate modulation weight  $v$ and study the corresponding boundedness properties  on modulation spaces $\M^{p,q}$ for $0<p,q\leq \infty$. The fractional powers of the anharmonic oscillator are well defined for any $\gamma\in\mathbb{R}$, since in the scale of  $S(m,g)$ classes the operator $\An$ turns out to be elliptic with respect to the metric $g=g^{k, \ell}$ in the sense of \cite{bufa:hy}. The main results in this work  extend the main  ones in \cite{than:hs, than:hs1} for the special case of the harmonic oscillator $\mathcal{H}=-\Delta+|x|^2=\Ann $.\\

At this point the reader might wonder: since the natural operator in the context of quantum mechanics is the unitary group $e^{-it\mathcal{H}}$ when $\mathcal{H}=-\Delta+ V(x)$, why to study the semigroup $e^{-t\mathcal{H}}$? Indeed, the usefulness of heat semigroups in quantum mechanics has been shown as long as 40 years ago and was well pointed out by B. Simon in \cite{BSimon1}.  For instance, when looking at the $L^p$ properties of the eigenfunctions of $\mathcal{H}=-\Delta+ V(x)$, if $\mathcal{H}\psi=E\psi$ and $\psi\in L^2$, one observes that since $e^{-t\mathcal{H}}\psi=e^{-tE}\psi$, then $\psi\in e^{-t\mathcal{H}}(L^2)$. Thus, a mapping property such as $e^{-t\mathcal{H}}:L^2\rightarrow L^{\infty}$ can be used to deduce  that $\psi$ belongs also to $ L^{\infty}$.   
The works of B. Simon and R. Carmona  on Schr\"odinger semigroups have deeply influenced the research on these semigroups, their  applications are diverse and the literature is currently huge.
See for instance \cite{albev:hs, Carmona1, BSimon2, BSimon3, iou:a1s, ishik:a1},  for more accounts on this matter. Thus, the Schr\"odinger semigroups are relevant for the study of $L^p$ properties of eigenfunctions of Schr\"odinger operators and especially for the so-called {\it ground state}.  By using them several quantities associated to the ground state can be analized, and in particular, they are useful in the quantum field theory \cite{Jaffe1, Simon1x}.
On the other hand, spectral properties for anharmonic oscillators have been investigated by many researchers for a long time. The case of polynomial potentials of even order has been studied by Helffer and Robert in \cite{hr:sap3, hr:anosc2, hr:anosc}; the one-dimensional anharmonic oscillator with non smooth potential $\mathcal{L}=-\frac{d^{2\ell}}{dx^{2\ell}}+|x|$ was recently considered in \cite{sikora:a1}.   Herein we will consider mapping properties between modulation spaces $\M^{p,q}$ for the semigroup corresponding to $\An$.\\%

The modulation spaces were introduced by H. Feichtinger in 1983 (cf. \cite{feich:mod}) and have been intensively investigated in the last decades. We refer the reader to  the survey \cite{feich:hist} 
by Feichtinger for a historical account on the development of such spaces and a good account of the literature.   
Modulation spaces have found  numerous applications in various problems in linear
and nonlinear partial differential equations, see \cite{RSW12}
for a recent survey. See also \cite{cp:cp1} for some recent applications of modulation spaces to nonlinear PDEs.
Here, in particular, we will apply our estimates
on modulation spaces to obtain the global well-posedness for a class
 of nonlinear parabolic equations associated with $\An^{\gamma}$.\\
 
The investigation of the global well-posedness for linear and nonlinear parabolic equations is an area of intensive research. 
 The well-posedness for a class of parabolic equations within the setting of Weyl-H\"ormander calculus has been studied  in \cite{Delgado2018}. In \cite{CDDF12, MiaoL}, the authors discussed  estimates of the heat propagator $e^{-t (-\Delta)^\beta}$ with $\beta>0$ of fractional Laplacian $(-\Delta)^\beta$ for modulation spaces and Lebesgue spaces on $\mathbb{R}^n$ and studied the global well-posedness of the Cauchy problem for the nonlinear fractional heat equation. In this case, the heat propagator is a Fourier multiplier on $\mathbb{R}^n$, therefore, using the standard results on the  boundedness of Fourier multipliers one can deduce the desired estimates. In the present paper, we are dealing with the anharmonic semigroups $e^{-t\An^{\gamma}}$ for $\gamma>0$ which are not, in general, Fourier multipliers. Therefore, we mainly depend on the Weyl-H\"ormander $S(M, g)$ symbol calculus. Our proof strategy shares similarities with that of \cite{than:hs, than:hs1}, but our proofs are significantly more intricate. This complexity arises due to the generality of the considered fractional anharmonic oscillators $\An^{\gamma}$. Recently, there has been a great interest in investigating the control problem involving the anharmonic oscillators and their fractional counterparts. In particular, the null-controllability and spectral inequalities for the evolution equations associated with the (fractional) anharmonic oscillators  were widely studied in \cite{AlphS, Alph, Martin} and references therein. It is interesting to record that the null-controllability of evolution equations involving the harmonic oscillator is different from the null-controllability of evolution equations involving  anharmonic oscillators (see \cite{brian, Martin}).  In the works \cite{CK, CK23}, the authors investigated the boundedness of spectral and Fourier multipliers associated with anharmonic oscillators.   \\

This paper is organized as follows. In Section \ref{preliminaries} we briefly  present the necessary background on modulation spaces and the $S(m,g)$ calculus. In Section \ref{Modul} we establish our main result regarding the mapping properties on modulation spaces. We start first by introducing a suitable modulation weight function and deducting the main mapping properties we are addressing. 
Finally, in Section \ref{nonl1} we present an application to the global well-posedness for non-linear heat equations involving fractional anharmonic oscillators on modulation spaces.

\section{Preliminaries: Modulation spaces and Weyl-H\"ormander calculus}\label{preliminaries}
In this section we briefly review some basics on modulation spaces and and the Weyl-H\"ormander calculus. We refer the reader to  Gr\"ochenig's book \cite{groch:bK} for the basic definitions and properties of modulation spaces. We also refer to 
\cite[Section 18.5]{ho:apde2}, \cite{ho:apde2, le:book, b-l:qua, Robert1} for a more detailed discussion on the Weyl-H\"ormander calculus.\\ 

We start by recalling the notion of modulation spaces and we refer the reader to the Chapter 11 of \cite{groch:bK} for a detailed exposition. 
A {\it modulation weight function} is a non-negative, locally integrable function on $\Rnn$. Since we will also consider the notion of weight associated to a H\"ormander metric, here in order to distinguish from them we additionally labeled the term {\it modulation} for the ones corresponding to the setting of modulation spaces.
%
A modulation weight function $v$ on $\Rnn$ is called {\it submultiplicative, } if
\begin{equation}
 \label{polyw0} 
v(x+y)\leq v(x)v(y) \mbox{ for all }x,y\in\Rnn . 
\end{equation}

A modulation weight function $w$ on $\Rnn$ is  {\it $v$-moderate, } if
\begin{equation}
w(x+y)\leq v(x)w(y) \mbox{ for all }x,y\in\Rnn . 
\end{equation}
In particular the weights of polynomial type play an important role. They are of the form
\begin{equation}\label{polyw} 
v_s(x,\xi)=(1+|x|^2+|\xi|^2)^{s/2}.
\end{equation}
The $v_{s}$-moderated weights (for some $s$) are called {\it polynomially moderated}.\\

Given a modulation weight $w$ on $\Rnn$, $0\leq p,q\leq \infty,$ and a window $g\in\mathcal{S}(\Rn),$ the modulation space $\M_{w}^{p,q}(\Rn)$ consists of the temperate distributions $f\in\mathcal{S} '(\Rn)$ such that

\begin{equation}\label{EQ:modul}
\Vert f \Vert_{{\M}_{w}^{p,q}}:=  \Vert V_{g}f \Vert_{L_w^{p,q}}:=
\left(\int_{\Rn}\left(\int_{\Rn}|V_gf(x,\xi)|^pw(x,\xi)^pdx\right)^{\frac{q}{p}}d\xi\right)^{\frac 1q}<\infty ,\end{equation} with suitable modification when $p=\infty$ or $q=\infty,$
where 
\begin{equation}\label{stft1}
V_gf(x,\xi)=  \int_{\Rn} f(y)\,\overline{g(y-x)}e^{- iy\cdot \xi} dy,
\end{equation}%
denotes the short-time Fourier transform of $f$ with respect to $g$ at the point $(x,\xi)$.  The modulation space $\M_{w}^{p,q}(\Rn)$ endowed with the above quasi-norm becomes a quasi-Banach space, independent of $g\neq 0$. For $p,q \geq 1,$ the modulation space $\M_{w}^{p,q}(\Rn)$ is a  Banach space and the quasi-norm becomes a norm. In this case, the dual of $\M_{w}^{p,q}(\Rn)$ is $\M_{\frac{1}{w}}^{p',q'}(\Rn),$ where $p'$ and $q'$ are the Lebesgue conjugates of $p$ and $q,$ respectively. If $p_1\leq p_2,\,q_1, \leq q_2$ and $w_1 \geq w_2$ then the embedding $\M_{w_1}^{p_1,q_1}(\Rn) \hookrightarrow \M_{w_2}^{p_2,q_2}(\Rn)$ holds.\\

We now recall basic notions of the Weyl-H\"ormander calculus and refer the reader to \cite{ho:apde2, le:book, b-l:qua} for more detailed discussions.\\

The Weyl quantisation of $a\in \mathscr{S}'(\mathbb{R}^n\times \mathbb{R}^n)$ is the operator $a^w(x,D):\mathscr{S}(\mathbb{R}^n)\rightarrow \mathscr{S}'(\mathbb{R}^n)$ defined by  
\begin{eqnarray*}
    a^{w}(x,D)u(x):=\frac{1}{(2\pi)^{n}}\int\limits_{\mathbb{R}^n}\int\limits_{\mathbb{R}^n}e^{ i\langle x-y,\xi \rangle}a((x+y)/2,\xi)u(y)dyd\xi,\,\,u\in \mathscr{S}(\mathbb{R}^n).
\end{eqnarray*}
More generally, we will consider a family of  $t$-quantisations  of  $a(\cdot,\cdot)$ and for any  $t\in\Ra,$ defined by
\begin{equation*}
    a_t(x,D)f(x):=\frac{1}{(2\pi)^{n}} \int\limits_{\mathbb{R}^n}\int\limits_{\mathbb{R}^n}e^{ i\langle x-y,\xi \rangle}a(t x+(1-t)y,\xi)u(y)dyd\xi.
\end{equation*}

Thus, the Weyl quantisation of $a(\cdot,\cdot)$ corresponds to the case $t=\half$ i.e. $a_{\frac{1}{2}}(x,D)=a^{w}(x,D).$ The case $t=1$ corrresponds to the Kohn-Nirenberg quantisation.
\begin{equation*}
    a(x,D)u(x)=a_1(x,D)u(x)= \frac{1}{(2\pi)^{n}}\int\limits_{\mathbb{R}^n}e^{ i\langle x,\xi \rangle}a(x,\xi)\widehat{u}(\xi)d\xi.
\end{equation*}
We now recall the definition of a H\"ormander metric. By looking at our anharmonic oscillators,  this concept will be crucial in particular to deduce some good properties for the fractional powers of $\An$.
\begin{definition}[H\"ormander's metric]\label{HM}
For $X \in \mathbb{R}^{2n}$, let $g_{X}(\cdot)$ be a positive definite quadratic form on $\mathbb{R}^{2n}$. We say that 
$g(\cdot)$ is a H\"ormander metric if the following three conditions are satisfied:
\begin{enumerate}
\item[I.] {\bf Continuity or slowness.} There exist a constant $C>0$ such that 
$$ g_{X}(X-Y)\leq C^{-1}\Longrightarrow \left(\frac{g_X(T)}{g_Y(T)}\right)^{\pm 1}\leq C,$$
for all $T\in\mathbb{R}^{2n}\setminus\{0\}$.
\item[II.] {\bf Uncertainty principle.} In terms of the symplectic form $\sigma(Y,Z):=\langle z,  \eta\rangle -\langle y, \zeta\rangle$,
we define $$ g_{X}^{\sigma}(T):=\sup_{W\neq 0}
\frac{\sigma(T,W)^{2}}{g_{X}(W)}. $$ We say that $g$ satisfies the {\it uncertainty principle }
if $$ \lambda_{g}(X)=\inf_{T\neq 0}
\left(\frac{g_{X}^{\sigma}(T)}{g_{X}(T)}\right)^{1/2}\geq 1 ,$$ for
all $X, T\in\mathbb{R}^{2n}$.\\

 
\item[III.] {\bf Temperateness.} We say  that $g$ is temperate if there exist $\overline{C}>0$ and $J\in \mathbb{N}$ such that
$$ \left( \frac{g_{X}(T)}{g_{Y}(T)}\right)^{\pm1}\leq \overline{C}(1+g_{Y}^{\sigma}(X-Y))^{J},
$$ for all $X,Y,T\in \mathbb{R}^{2n}$.
\end{enumerate}
\end{definition}


Let $g$ be a H{\"o}rmander metric.  {\it {The Planck function}} associated to $g$ is defined by
  \begin{equation*}\label{hgn1}
       h_g(X)^2=\sup _{T\neq 0} \frac{g_{X}(T)}{g_{X}^{\sigma}(T) }.
  \end{equation*} 
Note that $h_g(X)=(\lambda_{g}(X))^{-1}$. So,

\begin{definition}[$g$-weight]\label{GW}  Let $M:\mathbb{R}^{2n}\rightarrow \mathbb{R}^{+}$ be a function. Then,
\begin{itemize}
\item we say that 
$M$ is $g$-{{\it continuous}} if there exists $\tilde{C}>0$
such that
$$ g_{X}(X-Y)\leq \frac{1}{\tilde{C}}\Longrightarrow\left( \frac{M(X)}{M(Y)}\right)^{\pm1}\leq \tilde{C}.$$
\item we say that
$M$ is  $g$-{\it temperate} if there exist $\tilde{C}>0$ and $N\in \mathbb{N}$
such that
$$ \left( \frac{M(X)}{M(Y)}\right)^{\pm1}\leq \tilde{C}(1+g_{Y}^{\sigma}(X-Y))^{N}. $$
\end{itemize}
We will say that  $M$ is a $g$-{\it weight}  if it is $g$-continuous and $g$-temperate.
\end{definition}
Given a H\"ormander metric $g$ and a $g$-weight, we can now define the corresponding classes of symbols.  
\begin{definition} For a H{\"o}rmander metric $g$ and a $g$-weight $M$,  the class $S(M,g)$ consists of all smooth functions $\sigma\in C^\infty(\mathbb{R}^{2n})$ such that for any integer $k$ there exists $C_{k}>0$, such that for all
$X,T_{1},...,T_{k}\in \mathbb{R}^{2n}$ we have 
\begin{equation}\label{inwhk}
    |\sigma^{(k)}(X;T_{1},...,T_{k})|\leq C_{k}M(X)\prod_{i=1}^{k} g_{X}^{1/2}(T_{i}).
\end{equation}
The notation $\sigma^{(k)}$ stands for the $k^{th}$ derivative of $\sigma$ and  $\sigma^{(k)}(X;T_{1},...,T_{k})$ denotes the $k^{th}$  derivative of $a$ at $X$ with respect to  the directions $T_{1},...,T_{k}$. 
For $\sigma\in S(M,g)$ we denote by $\parallel \sigma\parallel_{k,S(M,g)}$ the minimum $C_{k}$ satisfying the above inequality. The class $S(M,g)$ becomes a Fr\'echet space endowed with the family of seminorms $\parallel \cdot\parallel_{k,S(M,g)}$.
\end{definition}

\section{Mapping properties of anharmonic fractional heat semigroup  on modulation spaces}\label{Modul}
In this section, we establish the boundedness on modulation spaces for a class of pseudo-differential operators including the negative powers of our anharmonic oscillators. We first introduce the specific setting regarding those spaces by choosing a suitable modulation weight function, and secondly an appropriate class of symbols by defining a H\"ormander metric adapted to the operators $\An$. \\

We start by introducing a modulation weight function on $\Rn\times\Rn$ for integers $k, \ell\geq 1$. First we define $\tilde{v}$  by  \[\tilde{v}(x,\xi):=(1+|x|^{k}+|\xi|^{\ell}),\]  
for all $x,\xi\in\Rn$. \\

We note that for all $x,y,\xi,\eta\in\Rn$ and a suitable $C>1$ we have 
\begin{align*}%
    \tilde{v}(x,\xi)\tilde{v}(y,\eta)=&\,(1+|x|^{k}+|\xi|^{\ell})(1+|y|^{k}+|\eta|^{\ell})\nonumber\\
    \geq &\, 1+|x|^{k}+|\xi|^{\ell} + 1+|y|^{k}+|\eta|^{\ell}\nonumber\\
        \geq &\, 1+|x|^{k}+|y|^{k}+|\xi|^{\ell}+|\eta|^{\ell}\nonumber\\
\geq & \,C^{-1}( 1+|x+y|^{k}+|\xi+\eta|^{\ell})\nonumber\\
    \geq & \,C^{-1}\tilde{v}(x+y,\xi + \eta).%
\end{align*}%
Hence, by defining $v=C\tilde{v}$, we obtain the  submultiplicativity of $v$. From now on we will only consider modulation spaces $\M_{w}^{p,q}$ corresponding to  modulation weights of the form $w=v^m$ for $m\in\Ra$. When $m=0,$ we simply write $\M^{p,q}.$\\

On the other hand, we will associate to the operator $\An$ the following metric on the phase-space:
\begin{equation}g:= g^{(k,\ell)}=\frac{dx^2}{(1+|x|^{2k}+|\xi|^{2\ell})^{\frac{1}{k}}}+\frac{d\xi ^2}{(1+|x|^{2k}+|\xi|^{2\ell})^{\frac{1}{\ell}}}.\label{anhmet012}
\end{equation}
The fact that this metric is indeed a H\"ormander metric has been proved in \cite{anh:cdr}, where the authors established spectral properties for the operators $\An$, by studying Schatten-von Neumann properties of their negative powers.  The fact that $1+|x|^{2k}+|\xi|^{2\ell}$ is a $g$-weight follows,  and in particular we have 
\[|x|^{2k}+|\xi|^{2\ell}\in S(1+|x|^{2k}+|\xi|^{2\ell}, g).\]

We observe that the uncertainty parameter for our metric is given by \begin{equation}
\label{unce56}\lambda_g=(1+|x|^{2k}+|\xi|^{2\ell})^{\frac{k+\ell}{2k\ell}}.    
\end{equation}
Therefore the {\it Planck function } $g$ and the modulation weight $v$ are related by 
\begin{equation}
h_g^{-1}=\lambda_g\simeq v^{\frac{k+\ell}{k\ell}}.
\end{equation}%
We note that the exponent  ${\frac{k+\ell}{2k\ell}}$ in \eqref{unce56} 
also arises in the study of asymptotic formulas for the eigenvalue counting function of our anharmonic oscillators, as we will clarify later on.\\

The H\"ormander symbol classes $S((1+|x|^{2k}+|\xi|^{2\ell})^{\frac{m}{2}},g)$ for $m\in\Ra$, can also be defined in a more intrinsic way:\\

 If $a\in C^{\infty}(\Rn\times\Rn)$ we will say that $a\in\Sigma_{k,\ell}^m$ if 
 \begin{equation}
\label{sigmacl}|\partial_{x}^{\beta}\partial_{\xi}^{\alpha}a(x,\xi)|\leq C_{\alpha\beta}(1+|x|^{k}+|\xi|^{\ell})^{m-\frac{|\beta|}{k}-\frac{|\alpha|}{\ell}}\,, \end{equation} holds for all multi-indices $\alpha, \beta$ and for all $x, \xi \in \mathbb{R}^n.$ This characterisation of $S((1+|x|^{2k}+|\xi|^{2\ell})^{\frac{m}{2}},g)=\Sigma_{k,\ell}^m$ was derived in \cite{anh:cdr}. \\

The space of symbols $\Sigma_{k,\ell}^m=S((1+|x|^{2k}+|\xi|^{2\ell})^{\frac{m}{2}},g)$ becomes a Fr\'echet space  by introducing the following seminorms:
\begin{equation} \Vert a\Vert _{\alpha, \beta}:=\inf\{C_{\alpha\beta}: \eqref{sigmacl}\,\,\mbox{ holds}\}.
\end{equation}

It will also be useful to point out that our symbol classes can also be embedded into the ones introduced by Nicola and Rodino in \cite{NicolaRodino}, by taking $k_0=\max\{k,\ell\}$. Indeed, we have
\begin{equation}g=g^{(k,\ell)}\leq \frac{dx^2}{(1+|x|^{2k}+|\xi|^{2\ell})^{\frac{1}{k_0}}}+\frac{d\xi ^2}{(1+|x|^{2k}+|\xi|^{2\ell})^{\frac{1}{k_0}}} .\label{anhmet012hh}
\end{equation}
Hence, by taking $\Phi(x,\xi)=\Psi(x,\xi)=(1+|x|^{2k}+|\xi|^{2\ell})^{\frac{1}{k_0}},\, $ and $M(x,\xi)=1+|x|^{2k}+|\xi|^{2\ell}$, it is not difficult to see that  $\Phi$ and $\Psi$ are sub-linear and temperate weights. Consequently, $M$ is a temperate weight. Then, by considering the  class $S(M;\Phi, \Psi)$ in the sense of \cite{NicolaRodino}, it is clear that,  \begin{equation}\label{ineqclassa}S((1+|x|^{2k}+|\xi|^{2\ell})^{\frac{s}{2}},g)\subset S(M^{\frac{s}{2}};\Phi, \Psi).\end{equation}

It is crucial to point out that the densely defined anharmonic oscillator $\mathcal{A}_{k, \ell}$ on $L^{2}(\Rn)$  is invertible. Indeed, since the operators $(-\Delta)^{\ell}$ and  $|x|^{2k}$ are strictly positive, then the eigenvalues of $\mathcal{A}_{k, \ell}$ are strictly positive. On the other hand, by Theorem 4.2.9 of \cite{NicolaRodino}, it follows that the spectrum of the closure $\overline{\mathcal{A}}_{k, \ell}$ in $L^{2}(\Rn)$ only consists of a discrete sequence of eigenvalues diverging to $\infty$. Therefore, $0$ does not belong to the spectrum of  $\overline{\mathcal{A}}_{k, \ell}$. Moreover, the eigenvalues of $\overline{\mathcal{A}}_{k, \ell}$ 
 have all finite multiplicity and the eigenfunctions belong to $\mathcal{S}(\Rn),$ and $L^{2}(\Rn)$ has an orthonormal basis consisting of eigenfunctions of $\overline{\mathcal{A}}_{k, \ell}$. \\

We also note that our anharmonic oscillators are $g$-elliptic in the sense of \cite{bufa:hy} and we can define the real powers $\mathcal{A}_{k, \ell}^{\gamma}$ for all $\gamma\in\Ra$ of the operators $\mathcal{A}_{k, \ell}$. The strong uncertainty principle in that setting also holds.\\

The Sobolev space $H_{k, \ell}^s(\mathbb{R}^n)$ for $s \in \Ra$ and for  $k, \ell \in \mathbb{N}$ is defined as 
$$ H_{k, \ell}^s(\mathbb{R}^n):=\left\{ u \in \mathcal{S}'(\Rn): \Vert (\mathcal{A}_{k, \ell})^{\frac{s}{2}} u\Vert_{L^2(\Rn)}<\infty \right\}.$$

The space $H_{k, \ell}^s(\mathbb{R}^n)$ will be called the anharmonic-Sobolev space of order $s$ relative to $k, \ell.$
For $k=1$ and $\ell=1,$ $H_{k, \ell}^m(\mathbb{R}^n)$ is the well-known Shubin-Sobolev space (also known as Hermite-Sobolev space).
The space $H_{k, \ell}^s(\mathbb{R}^n)$ is a Hilbert space equipped with the sesquilinear form 
$$\langle u,v \rangle:= \langle ( \mathcal{A}_{k, \ell})^{\frac{s}{2}}u,( \mathcal{A}_{k, \ell})^{\frac{s}{2}}v \rangle_{L^2(\Rn)}.$$
Also, the dual space of $H_{k, \ell}^s(\mathbb{R}^n)$ is $H_{k, \ell}^{-s}(\mathbb{R}^n).$ \\%

The following lemma characterises the anharmonic-Sobolev space. The proof of this lemma is verbatim to the proof of \cite[Lemma 4.4.19]{ElenaRodino} by using the weight $v^m$ and 
the embedding and duality properties of $H_{k, \ell}^m(\mathbb{R}^n).$ Therefore, we skip the proof.
\begin{lemma} \label{char:ansobo} For all $m \in \Ra,$ we have 
\begin{equation}
    \M_{v^m}^{2,2}=H_{k, \ell}^m(\mathbb{R}^n)
\end{equation} with equivalent norms.
\end{lemma}%
Using Lemma \ref{char:ansobo}, the inclusion relations given in Theorem 2.4.17 of \cite{ElenaRodino} and H\"older's inequality,  we deduce that for every $0<p,q \leq \infty$ and for $m \in \Ra$ large enough, we have%
\begin{equation} \label{embeddings}
    H_{k, \ell}^m(\mathbb{R}^n) \hookrightarrow \M^{p, q}(\Rn) \hookrightarrow \M^{\infty, \infty} \hookrightarrow H_{k, \ell}^{-m}(\mathbb{R}^n)
\end{equation}%
for all $k, l \geq 1.$ Recall that $\M^{p,q}(\mathbb{R}^n)$ is the modulation space $\M_{v^m}^{p,q}$ with $m=0.$\\%

We can now state the boundedness on modulation spaces for $t$-quantizations of symbols  in the setting of the above classes.%

\begin{theorem}\label{modb12} Let $m\in\Ra ,\, a\in\Sigma_{k,\ell}^{-m}$,  $0<p, q\leq\infty$ and  $t\in\Ra$. Then   
\[a_t(x,D):\M^{p, q}\rightarrow\M_{v^m}^{p,q},\] 
extends to a bounded operator, with operator norm depending only on a finite 
number of seminorms of $a$ in $\Sigma_{k,\ell}^{-m}$. 
\end{theorem}%
\begin{proof} We consider the following modulation weights  $\omega_1(x,\xi)=1$,  $\omega_2(x,\xi)=(1+|x|^{k}+|\xi|^{\ell})^m $  $(x,\xi\in\Rn)$ and $\omega_0(x,\xi,\eta,y)=(1+|x|^k+|\xi|^{\ell}+|y|^{k}+|\eta|^{\ell})^m ,$  $(x, \xi, y, \eta\in\Rn)$. An application of  Theorem 3.1 of \cite{toft1:mod} will give the boundedness of  
\[a_t(x,D):\M^{p,\, q}\rightarrow\M_{v^m}^{p,\,q},\]
for all $t\in\Ra$, provided $a\in \M_{\omega_0}^{\infty,\,  r}(\Rnn)$ for $r\leq\min\{1,p,q\}$.\\

In order to see that in fact $a\in \M_{\omega_0}^{\infty,\,  r}(\Rnn)$, for $W,Y\in\Rnn$, we first write: 
\begin{equation}\label{modexp1a}
e^{-iW\cdot Y}=(1+|w|^{2k}+|\tau|^{2\ell})^{-N}\left((\half +(-\Delta_{y})^k)^{N}+(\half +(-\Delta_{\eta})^{\ell})^{N}\right)e^{-iW\cdot Y},
\end{equation}%
where we have written $Y=(y,\eta), W=(w,\tau)$. \\%

We also note that for  a suitable $C>0$ we have
 \begin{equation}\label{in453}
 C^{-1}(1+|x|+|\xi|)^{\tau_0}\leq 1+|x|^{k}+|\xi|^{\ell}\leq C(1+|x|+|\xi|)^{\tau_1},
\end{equation}%
where $\tau_0=\min\{k, \ell\}$ and $\tau_1=\max\{k, \ell\}$. \\

We now plug  \eqref{modexp1a}  into the formula \eqref{stft1} for the short-time Fourier transform of the symbol $a$ on $\Rnn$. A repeated integration by parts and using the fact that $g\in \mathcal{S}(\Rnn) $ give us that the following holds: for every $N\in\ene$ there exists $C_N>0$ such that 
\begin{equation}\label{tt67}
                |V_ga(Z, W)|\leq C_N(1+|z|^{k}+|\gamma|^{\ell})^{-m}(1+|w|^{k}+|\tau|^{\ell})^{-2N},
\end{equation}
for all $Z, W\in\Rnn$ and where we have written $Z=(z,\gamma), W=(w,\tau)$.\\

From  \eqref{in453} and \eqref{tt67} we can conclude the proof by taking $N>\frac{n}{r\tau_0},$ which gives the finiteness of the desired modulation norm in $\M_{\omega_0}^{\infty,\,  r}(\Rnn)$ for the symbol $a$. 
\end{proof}

We now consider the special case of the positive powers of the anharmonic oscillator $\An$ and the corresponding semigroup. 
We will require some knowledge of the spectrum of the anharmonic oscillator $\An=(-\Delta)^{\ell}+|x|^{2k}$. We will denote by $\lambda_j$ the eigenvalues of $\An$ and arrange them in increasing order with the smallest eigenvalue denoted by $\lambda_0>0$. The estimates at the end of Section 5 in \cite{anh:cdr} give the asymptotics  for the eigenvalues $\lambda_j$:  
\begin{equation}\label{specA}
 \lambda_j\sim C_{k, \ell} j^{\frac{2k\ell}{n(k+\ell)}},\, \mbox{ as }\, j\rightarrow \infty.    
\end{equation}
We denote by $\Phi_j$ an orthonormal basis of eigenfunctions  corresponding to the eigenvalues  $\lambda_j$. We also denote by $d_j$ the dimension of the eigenspace $H_j$ corresponding to $\lambda_j$.   Thus, with respect to the orthogonal projections $P_j$ over each $H_j$ we  have
\[\An= \sum\limits_{j=0}^{\infty}\lambda_jP_jf, \hspace{ 0.4cm} P_jf=  \sum\limits_{i=1}^{d_j}\langle f,\Phi_{j_i}\rangle \Phi_{j_i} ,\]
where $\langle \cdot,\cdot\rangle $ denotes the inner product in  $L^2(\Rn)$.\\

We can define the fractional powers $\An^{\gamma}$ of $\An$ for any $\gamma\in\Ra$ by means of the spectral theorem
\[\An^{\gamma}f=\sum\limits_{j=0}^{\infty}\lambda_j^{\gamma}P_jf.\]
\\
Having this, we can say that the anharmonic-Sobolev space $H_{k, \ell}^s(\mathbb{R}^n)$ consists of all those functions $f \in \mathcal{S}'(\mathbb{R}^n)$ such that 
\begin{equation}
    \Vert f \Vert_{H_{k, \ell}^s(\mathbb{R}^n)}^2:= \Vert (\mathcal{A}_{k, \ell})^{\frac{s}{2}} u\Vert_{L^2(\Rn)}^2= \sum_{j=0}^\infty \lambda_j^s \Vert  P_jf \Vert_{L^2(\mathbb{R}^n)}^2 <\infty.
\end{equation}

The next result about the fractional anharmonic oscillator $\An^\gamma:=((-\Delta)^\ell+|x|^{2k})^\gamma$ will be essential in the proof of the main result of this section.
\begin{theorem}\label{regansymb}
Let $\gamma>0$ and let $k, \ell$ be integers $\geq 1.$ The fractional anharmonic oscillator $\An^\gamma:=((-\Delta)^\ell+|x|^{2k})^\gamma$ is a pseudo-differential operator with real Weyl symbol $a_{\gamma}:=  \An^\gamma(x, \xi) \in \Sigma^{2\gamma}_{k,\ell}.$ More precisely, the symbol $a_{\gamma} $ is given by
\begin{equation}\label{inet56}
    \An^\gamma(x, \xi):=(|\xi|^{2\ell}+|x|^{2k})^\gamma+r(x, \xi), \quad |\xi|^{\ell}+|x|^{k} \geq 1,
\end{equation} for some  
$r \in \Sigma^{2\gamma-(\frac{k+\ell}{k\ell})}_{k,\ell}.$
\end{theorem}%
\begin{proof} We first note that from \eqref{unce56},  the Planck function corresponding to our metric $g$ satisfies  \[h_g=(1+|x|^{2k}+|\xi|^{2\ell})^{-\frac{k+\ell}{2k\ell}}\leq C(1+|x|^{k}+|\xi|^{\ell})^{-\frac{k+\ell}{k\ell}}\leq C(1+|x|+|\xi|)^{-\frac{(k+\ell)\min\{k,\ell\}}{k\ell}},\] 
and the strong uncertainty principle holds for $g$ in the sense of \cite{NicolaRodino}.\\

On the other hand, since the H\"ormander metric $g$ is split and due to the form of the symbol $\sigma:=|\xi|^{2\ell}+|x|^{2k}$, it is not difficult to see that their corresponding $\sigma_t(x,D)$ quantizations coincide up to lower order terms by Theorem 2.3.12 of \cite{le:book}. The operator $\An$ is positive and Theorem 4.3.6 of \cite{NicolaRodino} applied to the class $S((1+|x|^{2k}+|\xi|^{2\ell})^{\frac{\gamma}{2}},g)$ regarding the positive powers of $\An^{\gamma}$, gives the desired result. Indeed, the Weyl symbol $a_{\gamma}=\An^{\gamma}(x,\xi)$ satisfies the following estimate for the first remainder \[\An^{\gamma}(x,\xi)-(|\xi|^{2\ell}+|x|^{2k})^\gamma\in \Sigma^{2\gamma-(\frac{k+\ell}{k\ell})}_{k,\ell}, \]
for $|\xi|^{\ell}+|x|^{k} \geq 1$. \end{proof}
As we have seen, the operator $A=\An^{\gamma}$ for $\gamma>0$, is a pseudo-differential operator with positive, elliptic Weyl symbol in $ S(M^{2\gamma};\Phi, \Psi)$ as in \eqref{ineqclassa} and its spectrum consists of a sequence of eigenvalues diverging to $\infty$. The corresponding eigenfunctions $\Phi_j$ define an orthonormal basis of $L^2(\Rn)$. We can then define, for $t\geq 0$ the {{\it   anharmonic  semigroup}} associated to $\An^{\gamma}$ for $\gamma>0$, by
\begin{equation}
e^{-t\An^{\gamma}}f=\sum\limits_{j=0}^{\infty}e^{-t\lambda_j^{\gamma}}P_jf.\end{equation}
Moreover, thanks to Theorem 4.5.1 of \cite{NicolaRodino}, the operator $e^{-t\An^{\gamma}}$ is a pseudo-differential operator with the Weyl symbol satisfying a family of inequalities that will be useful in the next theorem.\\

We can now state  the main result of this section.
\begin{theorem} \label{mainthmest}   Let $\gamma>0,$ $0<p_1,p_2,q_1,q_2 \leq \infty$ and let $k, \ell$ be integers $\geq 1.$ Define $\tilde{p}, \tilde{q}$ and $\sigma$ as follows:
$$\frac{1}{\tilde{p}}:=\text{max} \left\{ \frac{1}{p_2}-\frac{1}{p_1}, 0 \right\},\quad \frac{1}{\tilde{q}}:=\text{max} \left\{ \frac{1}{q_2}-\frac{1}{q_1}, 0 \right\}\,\,\text{ and }\,\, \sigma:= \frac{n}{2 \gamma} \Big(\frac{1}{k\tilde{p}}+\frac{1}{\ell\tilde{q}} \Big). $$

Then the  semigroup $e^{-t\An^{\gamma}}$ associated to the 
fractional anharmonic oscillator $\An^\gamma:=((-\Delta)^\ell+|x|^{2k})^\gamma$ on $\mathbb{R}^n$, satisfies the following estimate, for every $t>0,$
\begin{equation}\label{mappinganharmonic}
    \Vert e^{-t\An^{\gamma}} f\Vert_{{\M}^{p_2,q_2}(\mathbb{R}^n)} \leq C(t) \Vert  f\Vert_{{\M}^{p_1,q_1}(\mathbb{R}^n)},
\end{equation}
where
\begin{equation} \label{heatest}
    C(t)= C' \begin{cases} t^{-\sigma}\quad & 0<t \leq 1, \\ e^{-t \lambda_0^\gamma} \quad & t \geq 1,
    \end{cases} 
\end{equation}
for some positive constant $C',$  where $\lambda_0$ is the smallest eigenvalue of the anharmonic oscillator $\An.$
\end{theorem} 
\begin{proof} We divide the proof into two steps, in the first step we consider $0<t\leq 1$,  and after that the case  $t\geq 1.$ \\ %

\noindent {\bf Case $0<t \leq 1$}. The modulation spaces obey some inclusion relations that allow us to reduce the proof to consider $p_2$ replaced by $\min\{p_1, p_2\}$ and $q_2$ replaced by $\min\{q_1, q_2\}$. Thus, we will consider $p_2\leq p_1$ and $q_2\leq q_1$. \\%

On the other hand, since the operator $A=\An^{\gamma}$ for $\gamma>0$, is a pseudo-differential operator with positive, elliptic Weyl symbol in $ S(M^{2\gamma};\Phi, \Psi)$, its $g$-ellipticity following from  \eqref{inet56}. We can then  apply 
Theorem 4.5.1 of \cite{NicolaRodino} to the heat kernel 
$e^{-t\An^{\gamma}}$. The operator $e^{-t\An^{\gamma}}$ has a Weyl symbol $u(t,x,\xi)$ and we write $u_t(x,\xi)=u(t,x,\xi)$. For every $N\geq 0$, the symbol $t^Nu_t$ belongs to a bounded subset of $\Sigma_{k,\ell}^{-2\gamma N}$ provided that $t$ varies on a compact subset of $[0,+\infty)$.\\%

If $N\in\ene$ is such that $2\gamma N>2n$, applying Theorem \ref{modb12} to the symbols $u_t$ and $t^Nu_t$, with $\omega(x,\xi)=v_{2\gamma}(x,\xi)=(1+|x|^{k}+|\xi|^{\ell})^{2\gamma}$, we obtain
\[\Vert(1+t^N\omega^N)V_g(e^{-t\An^{\gamma}}f)\Vert_{L^{p_1,q_1}}\leq C \Vert f \Vert_{\M^{p_1, q_1}},\]
for a suitable constant $C$ independent of $t\in (0,1]$.\\%

Thus, in order to obtain \eqref{mappinganharmonic}, it will be enough to prove that 
 \begin{equation}\label{aux456}     
 \Vert F \Vert_{L^{p_2, q_2}}\leq C(t)\Vert (1+t^N\omega^N)F \Vert_{L^{p_1, q_1}},\end{equation}
with $C(t)$ as in  the statement and for all measurable functions $F$.\\%

By H\"older inequality and since $\frac{1}{p_2}=\frac{1}{p_1}+\frac{1}{\tilde{p}}$, $\frac{1}{q_2}=\frac{1}{q_1}+\frac{1}{\tilde{q}}$, we have
\begin{align*}\Vert F \Vert_{L^{p_2, q_2}}=& \Vert (1+t^N\omega^N)^{-1}(1+t^N\omega^N)F \Vert_{L^{p_2, q_2}}\\
\leq & \Vert (1+t^N\omega^N)^{-1}\Vert_{L^{\tilde{p}, \tilde{q}}}\Vert(1+t^N\omega^N)F \Vert_{L^{p_1, q_1}}.
\end{align*}

Now, to estimate $\tilde{C}(t)=\Vert (1+t^N\omega^N)^{-1}\Vert_{L^{\tilde{p}, \tilde{q}}}$, we proceed as follows:
\begin{align*}
    \tilde{C}(t)^{\tilde{q}}&=\Vert (1+t^N\omega^N)^{-1}\Vert_{L^{\tilde{p}, \tilde{q}}}^{\tilde{q}}=\smallint_{\mathbb{R}^n}\left(\smallint_{\mathbb{R}^n}(1+t^{N}(1+|x|^k+|\xi|^\ell)^{2\gamma N})^{-\tilde{p} }dx \right)^{\frac{\tilde{q}}{\tilde{p}} }d\xi\\
    &\leq \smallint_{\mathbb{R}^n}\left(\smallint_{\mathbb{R}^n}(1+t^{N}(|x|^k+|\xi|^\ell)^{2\gamma N})^{-\tilde{p} }dx \right)^{\frac{\tilde{q}}{\tilde{p}} }d\xi\\
    &= \smallint_{\mathbb{R}^n}\left(\smallint_{\mathbb{R}^n}(1+(t^{\frac{1}{2\gamma}}(|x|^k+|\xi|^\ell))^{2\gamma N})^{-\tilde{p} }dx \right)^{\frac{\tilde{q}}{\tilde{p}} }d\xi\\
     &= \smallint_{\mathbb{R}^n}\left(\smallint_{\mathbb{R}^n}(1+(|t^{\frac{1}{2\gamma k}}x|^k+|t^{\frac{1}{2\gamma \ell}}\xi|^\ell)^{2\gamma N})^{-\tilde{p} }dx \right)^{\frac{\tilde{q}}{\tilde{p}} }d\xi.
\end{align*}The change of variables $\tilde{x}=t^{\frac{1}{2\gamma k}}x,$ and $\tilde{\xi}= t^{\frac{1}{2\gamma \ell}}\xi$ imply that 
\begin{align*}
  \tilde{C}(t)^{\tilde{q}}& \leq  \smallint_{\mathbb{R}^n}\left(\smallint_{\mathbb{R}^n}(1+(|\tilde{x}|^k+|\tilde{\xi}|^\ell)^{2\gamma N})^{-\tilde{p} }   t^{-\frac{n}{2\gamma k} } d\tilde{x} \right)^{\frac{\tilde{q}}{\tilde{p}} } t^{ -\frac{n}{2\gamma\ell}  }d\tilde {\xi}  \\
  &= \smallint_{\mathbb{R}^n}\left(\smallint_{\mathbb{R}^n}(1+(|\tilde{x}|^k+|\tilde{\xi}|^\ell)^{2\gamma N})^{-\tilde{p} }    d\tilde{x} \right)^{\frac{\tilde{q}}{\tilde{p}} } d\tilde {\xi} \, t^{-\frac{n\tilde{q}}{2\gamma k\tilde{p} } } t^{ -\frac{n}{2\gamma\ell}  }.
\end{align*}
Note that with $N$ large enough, the integral 
$$ I^{\tilde{q}}:= \smallint_{\mathbb{R}^n}\left(\smallint_{\mathbb{R}^n}(1+(|\tilde{x}|^k+|\tilde{\xi}|^\ell)^{2\gamma N})^{-\tilde{p} }    d\tilde{x} \right)^{\frac{\tilde{q}}{\tilde{p}} } d\tilde {\xi} <\infty $$
is finite. So, with $N$ large enough we have proved that
\begin{align*}
    \tilde{C}(t)\leq I \times t^{-\frac{n}{2\gamma k\tilde{p} } } t^{ -\frac{n}{2\gamma\ell \tilde{q}}}\leq C' t^{-\frac{n}{2\gamma}(\frac{1}{k\tilde{p}}+ \frac{1}{\ell\tilde{q}})}\,=C(t), \quad 0<t \leq 1.
\end{align*}
This establishes \eqref{aux456} and hence the estimate \eqref{mappinganharmonic} for the case $0<t\leq 1.$
\\

\noindent{\bf Case $t \geq 1$}. In view of the following representation 
\[e^{-t\An^{\gamma}}f=\sum\limits_{j=0}^{\infty}e^{-t\lambda_j^{\gamma}}P_jf,\]
it is sufficient to show the following estimates in order to get estimate \eqref{mappinganharmonic}: 
\begin{equation} \label{eq3.15pro}
    \Vert P_j f\Vert_{M^{p_2,q_2} } \leq C'\lambda_j^{m} \Vert f \Vert_{M^{p_1, q_1}}. 
\end{equation}
for all $j \in \mathbb{N},$ for some $m \geq 0$ and $C'>0,$ and
\begin{equation} \label{eq3.6pro}
    \sum_{j=0}^{\infty} e^{-t \lambda_j^\gamma} \lambda_j^m \leq C'' e^{-t\lambda_0^{\gamma}}
\end{equation}
for all $t\geq 1$ and some positive constant $C''.$\\

To prove \eqref{eq3.15pro}, we will use the embedding \eqref{embeddings} along with the definition of anharmonic Sobolev space $H_{k, \ell}^m.$ In fact, by taking large enough $m,$ we obtain
\begin{equation}
    \Vert P_j \Vert_{M^{p_2,q_2} \rightarrow M^{p_1, q_1}} \leq C'\Vert P_j \Vert_{H_{k, \ell}^{-m} \rightarrow H_{k, \ell}^m}\leq   C'\lambda_j^m.
\end{equation}
Indeed, it is easy to see this by a simple calculation as follows:
\begin{align*}
    \Vert P_j f\Vert_{H^m_{k, \ell}}^2 &= \Vert \An^{\frac{m}{2}} P_j f \Vert_{L^2(\mathbb{R}^n)}^2=  \sum_{k=0}^\infty \lambda_k^{m} \Vert P_k(P_j f) \Vert_{L^2(\mathbb{R}^n)}^2\\&  =   \lambda_j^{m} \Vert(P_j f) \Vert_{L^2(\mathbb{R}^n)}^2 =\lambda_j^{2m}  \lambda_j^{-m} \Vert   (P_j f) \Vert_{L^2(\mathbb{R}^n)}^2 \\& \leq \lambda_j^{2m} \sum_{k=0}^\infty \lambda_k^{-m} \Vert P_k f \Vert_{L^2(\mathbb{R}^n)}^2 = \lambda_j^{2m} \Vert  f\Vert_{H^{-m}_{k, \ell}}^2.
\end{align*}
Next, we move forward to establish \eqref{eq3.6pro}. Recall that from \eqref{specA} that 
\begin{equation}
    \lambda_j\sim C_{k,\ell}  j^{\frac{2k\ell}{n(k+\ell)}},\, \mbox{ as }\, j\geq j'.
\end{equation}
Therefore, to make use of these asymptotics we prove \eqref{eq3.6pro} in two parts. First, note that  the function $j \mapsto e^{-t \lambda_j^\gamma} \lambda_j^m$ is a decreasing for, say, $j \geq j'',$ by noting that eigenvalues $\{\lambda_j\}_j$ are arranged in increasing order. Thus, we separately estimate, by setting $j_0:=\max\{j', \,j''\},$ that 
\begin{equation} \label{eq319}
    \sum_{j=0}^{j_0} e^{-t \lambda_j^\gamma} \lambda_j^m \leq C_1 e^{-t \lambda_0^\gamma}
\end{equation} for some positive constant $C_1.$
Also, for the remaining part
\begin{align*}
    \sum_{j=j_0+1}^\infty e^{-t \lambda_j^\gamma} \lambda_j^m &\leq  C_2 \sum_{j=j_0+1}^\infty e^{-t C_{k, \ell} j^{\frac{2k\ell \gamma}{n(k+\ell)}}}  j^{\frac{2k\ell m}{n(k+\ell)}} \\&  \leq C_2 \int_{j_0}^\infty e^{-t C_{k, \ell} x^{\frac{2k\ell \gamma}{n(k+\ell)}}}  x^{\frac{2k\ell m}{n(k+\ell)}} dx\\& = C_2 e^{-\lambda_0^\gamma} \int_{j_0}^\infty e^{-t (C_{k, \ell} x^{\frac{2k\ell \gamma}{n(k+\ell)}}-\lambda_0^\gamma)}  x^{\frac{2k\ell m}{n(k+\ell)}} dx.
\end{align*}
Now, applying the change of variable $C_{k, \ell}x^{\frac{2k\ell \gamma}{n(k+\ell)}}-\lambda_0^\gamma=y$ and observing that $C_{k, \ell} j_0^{\frac{2k\ell \gamma}{n(k+\ell)}}-\lambda_0^\gamma>0$ as $C_{k, \ell} j_0^{\frac{2k\ell }{n(k+\ell)}} \sim \lambda_{j_0}^{\gamma}>\lambda_0^{\gamma},$  we obtain 
\begin{align} \label{eq320}
    \nonumber  \sum_{j=j_0+1}^\infty e^{-t \lambda_j^\gamma} \lambda_j^m &\leq  C_2 \frac{n(k+\ell)}{2k \ell \gamma}e^{-t\lambda_0^\gamma} \int_{C_{k, \ell} j_0^{\frac{2k\ell \gamma}{n(k+\ell)}}-\lambda_0^\gamma}^\infty e^{-t y} (y+\lambda_0^{\gamma})^{\frac{m}{\gamma}+\frac{n(k+\ell)}{2k \ell \gamma}-1} dy \\& \nonumber \leq  C_2 \frac{n(k+\ell)}{2k\ell \gamma}e^{-t\lambda_0^\gamma} \int_{0}^\infty e^{-t y} (y+\lambda_0^{\gamma})^{\frac{m}{\gamma}+\frac{n(k+\ell)}{2k \ell \gamma}-1} dy\\&\leq C_2' e^{-t \lambda_0^\gamma} \int_{0}^\infty e^{- y} (y+\lambda_0^{\gamma})^{\frac{m}{\gamma}+\frac{n(k+\ell)}{2k \ell \gamma}-1} dy \leq C_2'\, e^{-t \lambda_0^\gamma}, 
\end{align}
where we have used the fact that the integral in penultimate inequality is decreasing in $t,$ so its value for $t\geq$ can not bigger than that of $t=1.$ Therefore, by combining \eqref{eq319} and \eqref{eq320}, we obtain the desired estimate \eqref{eq3.6pro}.
\end{proof}%

\section{Applications to global well-posedness of  the nonlinear heat equation for the anharmonic oscillator $\An^{\gamma}$}\label{nonl1}
In this section, our main aim is to investigate the global well-posedness of the following nonlinear fractional heat equation associated with the anharmonic oscillator $\An:$
\begin{equation}
    \begin{cases}
        \partial_t u+ \An^\gamma u=\lambda  |u|^{2\beta} u, \\
        u(0, x)=u_0(x),
    \end{cases}
\end{equation}
for $(t,x) \in (0, \infty) \times \mathbb{R}^n,$ where $\beta \in \mathbb{N},$ $\lambda \in \mathbb{C}$ and $\gamma>0.$\\

We begin by recalling some well-known properties of modulation space \cite{feich:mod,than:hs}. The following result is concerned with the algebra property of modulation spaces. 
\begin{lemma} \label{lm}
    Let $m \in \mathbb{N}$ and $p_i,q_j \in [0, 1]$ for $1\leq i \leq m$ such that $\sum_{i=1}^m \frac{1}{p_i}=\frac{1}{p_0}$ and $\sum_{i=1}^m \frac{1}{q_i}=m-1+\frac{1}{q_0}.$ Then, for some $C>0,$ we have
    $$\Bigg\Vert \prod_{i=1}^m f_i \Bigg\Vert_{\M^{p_0,q_0}(\mathbb{R}^n)}\leq C \prod_{i=1}^m \Vert f_i \Vert_{\M^{p_i,q_i}(\mathbb{R}^n)}.$$
\end{lemma}
The next result easily follows from Lemma \ref{lm} and the embedding of modulation spaces (for proof, see e.g. \cite{than:hs}). 
\begin{lemma} \label{multiesti} Let $p,q,r \in [1, \infty]$ and $\frac{1}{r}+2\beta=\frac{2\beta+1}{q}$ for $\beta \in \mathbb{N}.$
   Then the following multi-linear estimate holds:
$$\Vert |f|^{2\beta} f \Vert_{\M^{p, r}(\mathbb{R}^n)} \leq C \Vert f \Vert_{\M^{p,q}(\mathbb{R}^n)}^{2\beta+1}.$$
\end{lemma}

Now, we state and prove the main result of this section concerning  the global wellposedness of nonlinear heat equations associated with the fractional anharmonic oscillator.

\begin{theorem} Consider the nonlinear heat equation \begin{equation}
    \begin{cases} \label{nonpro}
        \partial_t u+ \An^\gamma u =\lambda    |u|^{2\beta} u, \\
        u(0, x)=u_0(x),
    \end{cases}
\end{equation}
for $(t,x) \in (0, \infty) \times \mathbb{R}^n,$ where $\beta \in \mathbb{N},$ $\lambda \in \mathbb{C}$ and $\gamma>0.$ Let $p,q \in [1, \infty]$ be such that  $2\beta+1 \leq q'$ and $\frac{\beta n}{\gamma \ell}<q'.$ Then, for a small initial data $u_0,$ that is, there exists $\epsilon>0$ such that $\Vert u_0 \Vert_{\M^{p,q}} \leq \epsilon,$ the problem \eqref{nonpro} admits a global solution $u$ in $L^\infty([0, \infty), \M^{p,q}).$ In addition to this, if $p<\infty,$ $u \in C([0, \infty), \M^{p,q}).$ \\

Moreover, if $\epsilon$ is sufficiently small then we get the exponential decay in time, in particular, $u \in Y,$ where the space $Y$ is defined as  \begin{equation} \label{spaceY}
    Y:=\left\{ u \in L^\infty([0, \infty), \M^{p,q}) \Big| \Big\Vert e^{t \lambda_0^\gamma}  \Vert u(t, \cdot) \Vert_{\M^{p,q}(\mathbb{R}^n)} \Big\Vert_{L^\infty_t([0, \infty))} <+\infty\right\}
\end{equation}
equipped with norm 
\begin{equation}
    \Vert u \Vert_Y := \Big\Vert e^{t \lambda_0^\gamma}  \Vert u(t, \cdot) \Vert_{\M^{p,q}(\mathbb{R}^n)} \Big\Vert_{L^\infty_t([0, \infty))},
\end{equation} where $\lambda_0$ is the smallest eigenvalue of the anharmonic oscillator $\An.$
    
\end{theorem}
\begin{proof}
    Using Duhamel's principle we can write the solution of $u$ of \eqref{nonpro} in the following integral form
    \begin{align} \label{eq4.5}
        u(t)= e^{-t\An^\gamma}u_0+\lambda \int_{0}^t e^{-(t-s)\An^\gamma} (|u(s)|^{2\beta} u(s))\,ds.
    \end{align}
    Now, with $p,q$ as in the statement of the theorem, from Theorem \ref{mainthmest} we have 
    \begin{align} \label{eq147}
        \Vert e^{-t\An^\gamma} f \Vert_{\M^{p,q}(\mathbb{R}^n)} \leq C' \Vert  f \Vert_{\M^{p,q}(\mathbb{R}^n)},
    \end{align}
    for all $t\geq 0$ and  some positive constant $C'.$  Also, for any $r\geq q$, by setting $p_1=p_2=p$ and $q_1=q \leq r=q_2$ and so $\sigma=\frac{n}{2\gamma \ell}\left(\frac{1}{q}-\frac{1}{r} \right)$ in Theorem \ref{mainthmest}, we have 
     \begin{align} \label{eq4.7}
        \Vert e^{-t\An^\gamma} f \Vert_{\M^{p,q}(\mathbb{R}^n)} \leq C' \begin{cases}
            t^{-\sigma} \Vert  f \Vert_{M^{p,r}(\mathbb{R}^n)} & \quad0<t \leq 1,\\
            e^{-t\lambda_0^\gamma} \Vert  f \Vert_{M^{p,r}(\mathbb{R}^n)} &\quad t \geq 1.
        \end{cases}  
    \end{align} for some $C'>0.$ 
    By the assumption $2\beta+1 \leq q',$ so we can pick $1\leq r\leq \infty$ such that  $\frac{2\beta+1}{q}=\frac{1}{r}+2\beta$ since $2\beta \leq \frac{2\beta+1}{q} \leq 2\beta+1$ and $\beta \in \mathbb{N}.$ Also, the assumption $\frac{\beta n}{\gamma \ell}<q'$ yields that $\sigma=\frac{n}{2\gamma \ell}\left(\frac{1}{q}-\frac{1}{r} \right)<1.$
    \\
    
    Let us now estimate the nonhomogeneous part of the integral \eqref{eq4.5}. An application of Minkowski's  inequality with  \eqref{eq4.7} gives
    \begin{align} \label{eq48}
   \nonumber     \Bigg\Vert \int_{0}^t e^{-(t-s)\An^\gamma}& (|u(s)|^{2\beta} u(s))\,ds \Bigg\Vert_{\M^{p,q}(\mathbb{R}^n)} \leq  \int_{0}^t \Vert e^{-(t-s)\An^\gamma} (|u(s)|^{2\beta} u(s)) \Vert_{\M^{p,q}(\mathbb{R}^n)}\,ds \\& \leq C'\begin{cases} \int_{0}^t
            (t-s)^{-\sigma} \Vert  (|u(s)|^{2\beta} u(s)) \Vert_{M^{p,r}(\mathbb{R}^n)}\,ds & \quad0<t \leq 1,\\
          \int_{0}^t  e^{-(t-s)\lambda_0^\gamma} \Vert (|u(s)|^{2\beta} u(s))  \Vert_{M^{p,r}(\mathbb{R}^n)}\,ds &\quad t \geq 1.\end{cases}  
    \end{align}
    Now, we use Lemma \ref{multiesti} to obtain
    \begin{align} \label{eq248}
    \nonumber    \Bigg\Vert \int_{0}^t e^{-(t-s)\An^\gamma}& (|u(s)|^{2\beta} u(s))\,ds \Bigg\Vert_{\M^{p,q}(\mathbb{R}^n)}  \nonumber\\& \leq C''\begin{cases} \int_{0}^t
            (t-s)^{-\sigma} \Vert   u(s) \Vert_{\M^{p,q}(\mathbb{R}^n)}^{2\beta+1}\,ds & \quad0<t \leq 1,\\
          \int_{0}^t  e^{-(t-s)\lambda_0^\gamma} \Vert  u(s)  \Vert_{\M^{p,q}(\mathbb{R}^n)}^{2 \beta+1}\,ds &\quad t \geq 1.\end{cases}  \nonumber \\&\leq C'' \Vert  u(s)  \Vert_{L^\infty([0, t], \,\M^{p,q}(\mathbb{R}^n))}^{2 \beta+1} \begin{cases} \int_{0}^t
            \tau^{-\sigma} \,ds & \quad0<t \leq 1,\\
          \int_{0}^t  e^{-\tau\lambda_0^\gamma} \,ds &\quad t \geq 1.\end{cases}  \nonumber\\& \leq C_1 \Vert  u(s)  \Vert_{L^\infty([0, \infty), \,\M^{p,q}(\mathbb{R}^n))}^{2 \beta+1},
    \end{align}
    because the integral given in the penultimate inequality is finite with the given condition $\sigma<1.$\\

As we are going to use the Banach contraction argument for the proof, let us set the operator of our interest
\begin{equation}
    \mathfrak{T}u:=e^{-t\An^\gamma}u_0+\lambda \int_{0}^t e^{-(t-s)\An^\gamma} (|u(s)|^{2\beta} u(s))\,ds
\end{equation}
for $u$ in the closed ball $$B_R:=\{u \in L^\infty([0, \infty), \,\M^{p,q}(\mathbb{R}^n)): \Vert u \Vert_{L^\infty([0, \infty), \,\M^{p,q}(\mathbb{R}^n))} \leq R\}$$
in $L^\infty([0, \infty), \,\M^{p,q}(\mathbb{R}^n))$ of radius $R$ and centered at origin. \\

{\bf Claim:} We claim that the operator $\mathfrak{T}$ is a contraction map from $B_R$ into $B_R$ for some suitable  $R>0.$

We first notice that it follows from \eqref{eq147} and \eqref{eq248} that for some $C_2>0,$ we have the following estimate for $\mathfrak{T}$ in $B_R,$
\begin{align}
     \Vert \mathfrak{T}u\Vert_{L^\infty([0, \infty), \,\M^{p,q}(\mathbb{R}^n))} &\leq C_2 \left( \Vert  u_0 \Vert_{\M^{p,q}(\mathbb{R}^n)}+ \Vert  u(s)  \Vert_{L^\infty([0, \infty), \,\M^{p,q}(\mathbb{R}^n))}^{2 \beta+1}\right),
\end{align}
which by assuming $\Vert  u_0 \Vert_{\M^{p,q}(\mathbb{R}^n)} \leq \frac{R}{2C_2},$ for now,  yields that 
\begin{align}
     \Vert \mathfrak{T}u\Vert_{L^\infty([0, \infty), \,\M^{p,q}(\mathbb{R}^n))} &\leq \frac{R}{2}+C_2 R^{2\beta+1},
\end{align}
as $u \in B_R.$
Now, to make sure that $\mathfrak{T}u \in B_R,$ we must choose $R$ such $C_2R^{2\beta+1}<\frac{R}{2},$ that is, $R< \left(\frac{1}{2C_2} \right)^{\frac{1}{2\beta}}.$ Therefore, choosing $\epsilon=\frac{R}{2C_2}$ with $R< \left(\frac{1}{2C_2} \right)^{\frac{1}{2\beta}}$ we have 
\begin{align}
     \Vert \mathfrak{T}u\Vert_{L^\infty([0, \infty), \,\M^{p,q}(\mathbb{R}^n))} &\leq \frac{R}{2}+\frac{R}{2}=R,
\end{align}
implying that $\mathfrak{T}u \in B_R.$

Next, we show that the map $\mathfrak{T}$ is a contraction. For this, by repeating  similar arguments as above for $u, v \in B_R$ and using the following inequality
$$|u|^{2\beta}u-|v|^{2\beta} v\leq C_3 (|u|^{2\beta-1}+|v|^{2\beta-1})|u-v|,$$
we have, for $C_4>0,$ that 
\begin{align}
 \nonumber   \Vert \mathfrak{T}u-\mathfrak{T}v\Vert_{L^\infty([0, \infty), \,\M^{p,q}(\mathbb{R}^n))} \leq &C_4 \left(\Vert u\Vert_{\M^{p,q}(\mathbb{R}^n)}^{2\beta-1}+ \Vert v\Vert_{\M^{p,q}(\mathbb{R}^n)}^{2\beta-1}\right)\Vert u-v\Vert_{\M^{p,q}(\mathbb{R}^n)}\\&\leq C_4'R^{2\beta-1} \Vert u-v\Vert_{\M^{p,q}(\mathbb{R}^n)}.
\end{align}
Hence, by taking $R>0$ sufficiently small and therefore $\epsilon$, for example $R<\left(\frac{1}{2C_4'}\right)^{\frac{1}{2\beta-1}},$ we conclude that $\mathfrak{T}$ is a contraction on $B_R,$ establishing our claim. 

Therefore, using Banach fixed point theorem, we conclude that $\mathfrak{T}$ has a unique fix point in $B_R$ which is the solution of nonlinear problem \eqref{nonpro}.
\\

  Additionally, let us assume that $p<\infty$. Our objective is to demonstrate that, under this assumption, the unique solution $u\in L^\infty([0, \infty), \M^{p,q}(\mathbb{R}^n))$ obtained from \eqref{nonpro} as described earlier is continuous with respect to $t$. This implies that $u \in C([0, \infty), \M^{p,q}(\mathbb{R}^n))$. The condition $2\beta+1\leq q'$ also implies that $q<\infty$. Let us assume for a moment that the semigroup ${e^{-t\An^\gamma}}$ is strongly continuous on $\M^{p,q}(\mathbb{R}^n),$ then we repeat the similar calculation and the Banach contraction argument as above by replacing $L^{\infty}([0, \infty), \M^{p,q}(\mathbb{R}^n))$ by $C([0, \infty), \M^{p,q}(\mathbb{R}^n))$ to get the desired conclusion. It is evident from \eqref{eq147} that  it is enough to show that, for every $f$ in the dense subset of $\M^{p,q}(\mathbb{R}^n),$ the map $t\mapsto e^{-t\An^\gamma} f$ is continuous with values in $\M^{p,q}(\mathbb{R}^n)$ in order to ensure that $e^{-t\An^\gamma}$ is a strongly continuous semigroup on $\M^{p,q}(\mathbb{R}^n).$ 

To establish the continuity of the map $t\mapsto e^{-t\An^\gamma} f$ with values in $\M^{p,q}(\mathbb{R}^n)$ for every $f$ within the dense subset of $\M^{p,q}(\mathbb{R}^n)$, we leverage the inclusion relation \eqref{anhmet012hh} and apply a certain abstract argument from \cite[Page 194]{NicolaRodino}. This involves utilizing the fact that the semigroup $e^{-t\An^\gamma}$ exhibits strong continuity on $L^2(\mathbb{R}^n)$ and $\An^\beta$ commutes with $e^{-t\An^\gamma}$ for every $\beta \in \mathbb{N}$. Consequently, the mapping $t \mapsto e^{-t\An^\gamma}\An^\beta f=\An^\beta e^{-t\An^\gamma}f$ is continuous with values in $L^2(\mathbb{R}^n)$ for every $\beta \in \mathbb{N}$.  
By referring to \cite[Page 194]{NicolaRodino}, we infer that the seminorms $p_{\beta}(f):=\Vert \An^\beta f\Vert_{L^2(\mathbb{R}^n)}$ for $\beta \in \mathbb{N}$ define an equivalent family of seminorms on $\mathcal{S}(\mathbb{R}^n)$. Consequently, the mapping $t \mapsto e^{-t\An^\gamma} f$ exhibits continuity with values in $\mathcal{S}(\mathbb{R}^n)$, and by extension, when treated as a $\M^{p,q}(\mathbb{R}^n)$-valued function. This conclude that for $p<\infty,$ the unique solution $u$ of \eqref{nonpro} is in  $ C([0, \infty), \M^{p,q}(\mathbb{R}^n).$\\

    Now we proceed to deduce that by choosing $\epsilon>0$ sufficiently small we can get the desired rate of decay of $u.$ Again, from Theorem \ref{mainthmest} it follows that 
$$\Vert e^{-t\An^\gamma} f \Vert_{\M^{p,q}(\mathbb{R}^n)} \leq C_0 e^{-t \lambda_0^\gamma} \Vert f\Vert_{\M^{p,q}(\mathbb{R}^n)},$$
for some $C_0>0,$ which implies that 
\begin{equation} \label{M1}
    \Vert e^{-t\An^\gamma} f \Vert_{Y} \leq C_0 \Vert f\Vert_{\M^{p,q}(\mathbb{R}^n)},
\end{equation}
where the space $Y$ is defined by \eqref{spaceY}. 

Similar to estimate \eqref{eq48},  we have

\begin{equation} \label{M2}
    e^{t\lambda_0^\gamma}\Bigg\Vert \int_0^t e^{-(t-s)\An^\gamma} (|u(s)|^{2\beta}u(s))\, ds \Bigg\Vert_{\M^{p,q}(\mathbb{R}^n)} \leq C_0'e^{t\lambda_0^\gamma} \int_0^t C(t-s) \Vert u(s) \Vert_{\M^{p,q}(\mathbb{R}^n)}^{2\beta+1} \,ds,
\end{equation} where $C(t)$ for $t>0$ is given by \begin{equation} \label{heatest1}
    C(t)= C' \begin{cases} t^{-\sigma}\quad & 0<t \leq 1, \\ e^{-t \lambda_0^\gamma} \quad & t \geq 1,
    \end{cases} 
\end{equation}
for some positive constant $C'.$ To estimate the integral on the right-hand side, we set 
$T_1:=\{s \in [0, t]: t-s<1\}$ and $T_2:=\{s\in [0, t]: t-s \geq 1\}.$
Therefore, we obtain
\begin{align} \label{M3}
 \nonumber   &e^{t\lambda_0^\gamma} \int_0^t C(t-s) \Vert u(s) \Vert_{\M^{p,q}(\mathbb{R}^n)}^{2\beta+1} \,ds \\&\quad= e^{t\lambda_0^\gamma} \int_{T_1} C(t-s) \Vert u(s) \Vert_{\M^{p,q}(\mathbb{R}^n)}^{2\beta+1} \,ds +e^{t\lambda_0^\gamma} \int_{T_2} C(t-s) \Vert u(s) \Vert_{\M^{p,q}(\mathbb{R}^n)}^{2\beta+1} \,ds 
   \nonumber \\&\leq C_1 e^{t\lambda_0^\gamma} \int_{T_1} (t-s)^{-\sigma} (e^{-s\lambda_0^\gamma} e^{-s\lambda_0^\gamma}\Vert u(s) \Vert_{\M^{p,q}(\mathbb{R}^n)})^{2\beta+1} \,ds \nonumber\\& \quad\quad\quad\quad\quad\quad\quad\quad\quad+ C_1e^{t\lambda_0^\gamma} \int_{T_2} e^{-(t-s)\lambda_0^\gamma} \Vert u(s) \Vert_{\M^{p,q}(\mathbb{R}^n)}^{2\beta+1} \,ds \nonumber\\&\leq C_1 e^{t\lambda_0^\gamma} \int_{T_1} (t-s)^{-\sigma} e^{-s(2\beta+1)\lambda_0^\gamma} ( e^{s\lambda_0^\gamma}\Vert u(s) \Vert_{\M^{p,q}(\mathbb{R}^n)})^{2\beta+1} \,ds \nonumber\\& \quad\quad\quad\quad\quad\quad\quad\quad\quad+ C_1 \int_{T_2} e^{2\beta s\lambda_0^\gamma}(e^{s\lambda_0^\gamma} \Vert u(s) \Vert_{\M^{p,q}(\mathbb{R}^n)})^{2\beta+1} \,ds \nonumber\\& \leq C_1 e^{t\lambda_0^\gamma} \Vert u(s) \Vert_{Y}^{2\beta+1} \int_{T_1} (t-s)^{-\sigma} e^{-s(2\beta+1)\lambda_0^\gamma} \,ds+C_1  \Vert u(s) \Vert_{Y}^{2\beta+1} \int_{T_2} e^{-2\beta s\lambda_0^\beta}\,ds \nonumber\\& \leq C_1 e^{t\lambda_0^\gamma} \Vert u(s) \Vert_{Y}^{2\beta+1} \int_{0}^1 \tau^{-\sigma}  \,d\tau +C_1  \Vert u(s) \Vert_{Y}^{2\beta+1} \int_{0}^\infty e^{-2\beta s\lambda_0^\beta}\,ds \nonumber\\& \leq C_1'  \Vert u(s) \Vert_{Y}^{2\beta+1}, 
\end{align}
where we have used that $\beta >0$ and $\sigma<1$ by the hypothesis in the statement of the theorem.
Therefore, putting together \eqref{M1}, \eqref{M2} and \eqref{M3}, we have
\begin{align*}
   \Vert \mathfrak{T}u \Vert_{Y} &\leq  \Vert e^{-t\An^\gamma}u_0 \Vert_{Y} +\Big\Vert \lambda \int_{0}^t e^{-(t-s)\An^\gamma} (|u(s)|^{2\beta} u(s))\,ds  \Big\Vert_Y \\& \leq C_2 \left( \Vert u_0 \Vert_{\M^{p,q}(\mathbb{R}^n)}+ \Vert u(s) \Vert_{Y}^{2\beta+1} \right).
\end{align*}
Finally, we can argue as before using Banach's contraction argument to get the desired result. \end{proof}

\bibliographystyle{amsplain}

\end{document}